\documentclass[11pt,letterpaper]{amsart}
\allowdisplaybreaks 
\belowdisplayskip
\abovedisplayskip

\usepackage{pgfplots}
\pgfplotsset{compat=1.18}  
\usetikzlibrary{intersections} 
\usepgfplotslibrary{fillbetween}
\usepackage{changes}

\usepackage{lmodern}
\usepackage[T1]{fontenc}

\makeatletter
\@namedef{subjclassname@2020}{%
  \textup{2020} Mathematics Subject Classification}
\makeatother
\usepackage[table,xcdraw]{xcolor}
\usepackage[a4paper,top=2.5cm,bottom=2.5 cm,left=2 cm,right=2 cm]{geometry}

\usepackage[utf8]{inputenc} 
\usepackage[english]{babel} 
\usepackage{fancyhdr}
\usepackage[toc,page]{appendix}
\usepackage{amsmath,amssymb,mathrsfs,mathrsfs,pifont,cancel,paralist,mathbbol,bbm}
\usepackage{amsthm}
\usepackage{mathtools}
\usepackage{caption, subcaption}
\usepackage{tabularx}
\usepackage{graphicx}
\usepackage{mathrsfs}
\usepackage{pdfcomment}
\usepackage{url}
\hypersetup{hidelinks}
\usepackage[a-1b]{pdfx}
\usepackage[T1]{fontenc}
\usepackage{bm}
\usepackage{enumitem}
\usepackage{comment}
\usepackage[all]{xy}
\usepackage{tikz}
\usepackage{tikz-cd}
\usepackage{mathbbol}
\usepackage{epigraph}
\usepackage{setspace}
\usepackage{ stmaryrd }
\usepackage{nicefrac,xfrac}
\usepackage{amsmath}
\usepackage{makecell} 
\usepackage{hyperref}

\definecolor{forestgreen(web)}{rgb}{0.13, 0.55, 0.13}

\newcommand{\commentoo}[1]{}

\newcommand{\N}{\mathbb{N}}
\newcommand{\Z}{\mathbb{Z}}

\newcommand{\R}{\mathbb{R}}
\newcommand{\C}{\mathbb{C}}
\renewcommand{\P}{\mathbb{P}}
\newcommand{\mfS}{\mathfrak{S}}
    
\newcommand{\mcal}{\mathcal}

\newcommand{\pol}[2]{\mathrm{Pol}_{#1}(#2)}

\newcommand{\ord}[1]{{\left| {#1}\right|}}

\newcommand{\psiabs}{\psi^{\mathrm{abs}}}
\theoremstyle{definition}
\newtheorem{defi}{Definition}[section]

\newtheorem{remark}[defi]{Remark}

\newtheorem*{notation}{Notation}
\newtheorem*{comm}{Comment}

\theoremstyle{plain}
\newtheorem{prop}[defi]{Proposition}

\newtheorem{lemma}[defi]{Lemma}
\newtheorem{theorem}[defi]{Theorem}

\title{Worse than square-root cancellation in Bateman--Horn's conjecture}
\author{Giacomo Bortolussi}
\address{Università di Roma Tor Vergata\\ Dipartimento di Matematica\\ 00133\\ Rome}
\email{bortolussi@axp.mat.uniroma2.it}
\begin{document}
\begin{abstract}
    We prove asymptotics for the average error term in Bateman--Horn's conjecture in the exponential range.
\end{abstract}
\maketitle

\setcounter{tocdepth}{1}
\tableofcontents

\section{Introduction}
 
Let $P\in \Z[t]$ be an irreducible polynomial with a positive leading coefficient. Let $\Lambda$ be the von Mangoldt function and let  $$\psi_P(x)\coloneqq\sum_{\substack{1\le n\le x\\P(n)>0}} \Lambda(P(n))\,.$$
The Bateman--Horn conjecture~\cite{BH} states that $\psi_P(x)=\mfS_Px+o(x)$ as $x\to \infty$,  where 
\begin{equation*}\mfS_P\coloneqq \prod_{\substack{\ell =2\\\ell\text{ prime}}}^\infty\left(1-\frac 1 \ell\right)^{-1}\left(1-\frac{\#\{m\in \Z/\ell\Z :P(m)=0 \}} \ell\right)\,.\end{equation*}
When $\deg P=1$, the conjecture reduces to Dirichlet's theorem on arithmetic progressions, but there is no $P$ with $\deg P\ge 2$ for which the conjecture is established. In light of this, there has been recent work on \lq\lq average\rq\rq\, versions. For fixed $d\in \N$ and any $H\ge 1$, let
 $$\pol{d}{H}\coloneqq\left\{P(t)=c_dt^d+\ldots+c_0\in \Z[t]: (c_d,\ldots,c_0)\in [-H,H]^{d+1},\, c_d>0\right\}.$$
 For $z>1$, define the truncated Bateman--Horn constant as \begin{equation}\label{def_tbh}
     \mfS_P(z)\coloneqq \prod_{\substack{ \ell < z\\\ell\text{ prime}}}\left(1-\frac 1 \ell\right)^{-1}\left(1-\frac{\#\{m\in \Z/\ell\Z :P(m)=0 \}}\ell\right)\,.
 \end{equation}
 We show that square-root cancellation in Bateman--Horn's conjecture fails on average for $x$ in a certain range. For example, taking $x=(\log H)^2$ and averaging over $P\in \pol{d}{H}$, Theorem~\ref{teo_main} gives $$\ord{\psi_P(x)-x\mfS_P(x)}\asymp \sqrt{x\log H}= x^{\frac{3}{4}}\,.$$

\begin{theorem}
    Fix $d\ge 2$ and $\delta>0$. For all $x,H>2$ with $x\le (\log H)^\delta$ and $x=x(H)\to \infty$, we have
    \begin{align*}\frac 1 {2^dH^{d+1}}\!\sum_{P\in \pol{d}{H}}\!\!\ord{\psi_P(x)-x\mfS_P(x)}^2=x\log H+O\left(x\sqrt{(\log H)(\log x)}+(\log H)\sqrt{x(\log x)}\right)\,,\end{align*}
    where the implied constant depends only on $d$ and $\delta$.
\label{teo_main} \end{theorem}
\begin{remark}
For general polynomials $P$ there is scarce
  evidence for the true order of magnitude of $\ord{\psi_P(x)-\mfS_Px}$.
  When $\deg P=1$, there is extra structure coming from the zeroes of Dirichlet \mbox{$L$-functions} and there are conjectural error terms arising from  
  the Generalised Riemann Hypothesis and Montgomery's Conjecture \cite[Conjecture 1(b)]{FG1}. For polynomials of the form $P(t)=t^d+at+b$, Friedlander--Granville \cite{FG} proved that for some $x>(\log \ord P)^N$, the inequality $$\ord{\psi_P(x)-x\mfS_P}>\delta_N\mfS_Px$$  holds for some constant $\delta_N>0$ depending only on $N$. On the contrary, Theorem~\ref{teo_main} gives evidence for the error term in the Bateman--Horn conjecture for \textit{typical} polynomials.
\end{remark}

Our work makes crucial use of Hooley's neutralisers to average $\mfS_P$. In doing so, we improve on the work of Skorobogatov--Sofos~\cite[\S 4]{SS}, who used the circle method to prove the bound $O\left(\frac {x^2}{\log x}\right)$ for the left-hand side in Theorem \ref{teo_main}, but with $\mfS_P(\log x)$ in place of $\mfS_P(x)$.
Note that they work with $$\theta_P(x)\coloneqq \sum_{\substack{1\le n\le x\\P(n)\text { prime}}}\log P(n)$$ in place of $\psi_P$, but it is straightforward to pass from $\theta_P(x)$ to $\psi_P(x)$ without changing their upper bound.
    Theorem~\ref{teo_main} improves on the bound of Skorobogatov--Sofos; in particular, it also recovers the parts of~\cite{SS} related to Schinzel's hypothesis.

Recently, Kravitz--Woo--Xu~\cite[Theorem~1.1]{Wo} applied a Green--Tao-type result of Leng~\cite{Le} to prove 
\begin{equation}\label{eq-woo}\frac{1}{H^{d+1}}\sum_{P\in \pol d H}\ord{\psi_P(x)-x\mfS_P'}^k\ll \frac {x^k}{(\log x)^{1+o(1)}}\,,\end{equation}
where $\mfS_P'$ is a suitably truncated version of $\mfS_P$ and with $k\in \N$ and $x=(\log H)^\delta$ for a fixed $\delta>1$.
\begin{remark}
    Theorem~\ref{teo_main} gives an asymptotic for the second moment, whereas (\ref{eq-woo}) gives an upper bound. By Cauchy's inequality and Theorem~\ref{teo_main}, we have 
    $$\frac{1}{H^{d+1}}\sum_{P\in \pol d H}\ord{\psi_P(x)-x\mfS_P(x)}\ll \sqrt{x\log H}\,. $$
    For $k=1$, this is an improvement on (\ref{eq-woo}) when $\delta>2$. Furthermore, the trivial bounds $\psi_P(x)\ll x\log H$ and $\mfS_P(x)\ll \log H$ together with Theorem~\ref{teo_main} that imply for $k\ge 3$
    $$\frac{1}{H^{d+1}}\sum_{P\in \pol d H}\ord{\psi_P(x)-x\mfS_P(x)}^k\ll {(x\log H)^{k-1}}\,,$$
    which improves upon~(\ref{eq-woo}) for all $k\ge 3$ and $\delta>k$.
\end{remark}
It would be interesting to extend Theorem \ref{teo_main} to the range $x\le H^{\alpha}$ (as done in \cite{BST}), but one would need a substitute for Hooley's neutralisers.


\begin{notation}
    The letter $\ell$ is always used to indicate a prime number. Given $P\in \Z[t]$ and $m\in \N$, we define $\omega_P(m)\coloneqq\#\{x\in \Z/m\Z: P(x)=0\}$. Given $m\in \N$, we define $\omega(m)$ to be the number of distinct prime divisors of $m$.
\end{notation}


\section{Strategy of the proof}

A crucial feature of previous work~\cite{Wo,SS} is the truncation of $\mfS_P(z)$ at a small value $z=z(x)$. Skorobogatov--Sofos~\cite{SS} chose $z=\log x$, which inherently leads to an error term of size $O\left(\frac{x^2}{\log x}\right)$. To obtain asymptotics for the second moment within their method, it seems necessary to truncate $\mfS_P$ at $z=x^\gamma$ with $\gamma\ge 1$. Building on the method of Skorobogatov--Sofos, we attain such a larger truncation by employing Hooley's neutralisers~\cite{Ho} combined with Brun's sieve~\cite[Lemma 6.4]{IK}.

It would be interesting to combine \cite[Theorem 3.2]{Wo} with Hooley's neutralisers to obtain asymptotics for all moments, but it seems that one would need to take $M\le H^A$ instead of $M\le (\log H)^A$.

We define
\begin{equation}\label{def_psiabs}
\psiabs_P(x)\coloneqq \sum_{\substack{1<m\le x\\P(m)\neq 0}}\Lambda(\ord{P(m)}).
\end{equation}

\begin{theorem}
    \label{teo_mainabs} Fix $d\ge 2$, $\delta>0$, and $\gamma\ge 1$. Then for all $x,H>2$ with $x\le (\log H)^\delta$ we have
    \begin{equation*}
        \frac 1 {2^dH^{d+1}}\sum_{P\in \pol d H}\ord{\psiabs_P(x)-x\mfS_P(x^\gamma)}^2=x\log H+O(x\log\log H)\,.
    \end{equation*}
    where the implied constant depends only on $d$, $\delta$, and $\gamma$.
\end{theorem}

Working with $\psiabs_P(x)$ instead of $\psi_P(x)$ temporarily bypasses the somewhat inconvenient condition $P(m)>0$ in $\psi_P$.
In \S \ref{sec_trans}, we explain how Theorem \ref{teo_main} follows from Theorem \ref{teo_mainabs}.

To prove Theorem \ref{teo_mainabs}, we write the desired sum as
\begin{equation}\label{sq_opsq}
    \sum\limits_{P\in\pol{d}{H}}\psiabs_P(x)^2-2x\sum\limits_{{P\in\pol{d}{H}}}\psiabs_P(x)\mathfrak S_P(z)+x^2\sum\limits_{P\in\pol{d}{H}}\mathfrak S_P(z)^2
\end{equation}
and prove the following propositions.
\begin{prop}\label{prop_est_psi_sq} Keep the setting of Theorem~\ref{teo_mainabs} and let $\mcal L>1$. Then
    \begin{equation*}
        \frac 1{2^dH^{d+1}}\sum_{P\in \pol d H}\psiabs_P(x)^2=x^2\prod\limits_{\substack{\ell \text{ prime}\\\ell<\mcal L}}\left(1+\frac 1{\ell(\ell-1)}\right)+x(\log H)+O\left( x(\log x)+\frac
    {x^2}{\mcal L}\right)\,,
    \end{equation*}
    where the implied constant depends only on $d$ and $\delta$.
\end{prop}
\begin{prop}\label{prop_est_Ssq} Keep the setting of Theorem~\ref{teo_mainabs} and fix $A>0$. Then for $2\le z\le \frac{H^{1/2}}{(\log H)^A}$ we have
    $$\frac 1 {2^dH^{d+1}}\sum_{P\in \pol d H}\mathfrak S_P^2(z)=\prod_{\substack{\ell<z\\\ell\text{ prime}}}\left(1+\frac{1}{\ell(\ell-1)}\right)+O\left((\log z)^2H^{-\frac{1}{2\log z}}+\frac{(\log z)^2}{(\log H)^{2A-1}}\right)\,,$$
    where the implied constant depends only on $d, \delta$ and $A$.
\end{prop}
\begin{prop}\label{prop_est_psiS} Keep the setting of Theorem~\ref{teo_mainabs} and fix $A>5$. For $2\le z\le \frac{H^{1/2}}{(\log H)^{A}}$ we have
    $$\frac 1 {2^dH^{d+1}}\sum_{P\in \pol d H}\mathfrak S_P(z)\psiabs_P(x)=x\prod_{\substack{\ell< z\\\ell\text{ prime}}}\left(1+\frac 1 {\ell(\ell-1)}\right)+O\left(x\,{(\log z)}\,{H^{-\frac 1 {2\log z}}}+\frac{x\log z}{(\log H)^{A-d\delta-5}}\right)\,,$$
    where the implied constant depends only on $d,\delta$ and $A$.
\end{prop}

The proof of Proposition \ref{prop_est_psi_sq} is given in \S\ref{sec_psisq} and relies on Skorobogatov--Sofos' treatment of the analogous sum. By further opening up the squares, \begin{align*}
    \frac 1{2^dH^{d+1}}\sum_{P\in \pol d H}\psiabs_P(x)^2=&\frac 1{2^dH^{d+1}}\sum_{\substack{m\le x}}\sum_{\substack{P\in \pol d H\\P(m)\neq 0}}\Lambda(\ord{P(m)})^2\\
    &+\frac 1{2^dH^{d+1}}\sum_{\substack{m_1\neq m_2\le x}}\sum_{\substack{P\in \pol d H\\P(m_1),P(m_2)\neq0}}\Lambda(\ord{P(m_1)})\Lambda(\ord{P(m_2)})\,.
\end{align*}
We give asymptotics for the diagonal term $\sum\limits_{m\le x}\sum\limits_{P\in \pol d H}\Lambda(\ord{P(m)})^2$, whereas \cite{SS} provides only a bound.
The non-diagonal term $\sum\limits_{m_1\neq m_2\le x}\sum\limits_{P\in \pol d H}\Lambda(\ord{P(m_1)})\Lambda(\ord{P(m_2)})$ contributes the main term $x^2\prod\limits_{\ell=2}^\infty \left(1+\frac 1 {\ell(\ell-1)}\right)$.
This agrees with the main terms of $\sum_P\psiabs_P(x)\mfS_P$ and $\sum_P\mfS_P^2$ only after truncating the infinite product at $\ell<\mcal L$. This truncation yields the error term $O\left(\frac{x^2}{\mcal L}\right)$ in Proposition \ref{prop_est_psi_sq}, which is smaller than the secondary term $x\log H$ only when $\mcal L\ge \frac x {\log H}$. Thus, for the main terms of Propositions \ref{prop_est_psi_sq}-\ref{prop_est_psiS} to match, one must take a large truncation point $z$ in (\ref{def_tbh}). Consequently, the proofs of the statements analogous to Propositions \ref{prop_est_Ssq}-\ref{prop_est_psiS} in \cite{Wo,SS} no longer apply. It is at this point that we bring in Hooley's neutralisers. The proofs of Propositions \ref{prop_est_Ssq}-\ref{prop_est_psiS} are given, respectively, in \S\S \ref{sec_Ssq}-\ref{sec_psiS}.

\begin{proof}[Proof of Theorem \ref{teo_mainabs}]
    It suffices to substitute the asymptotics of Propositions \ref{prop_est_psi_sq}-\ref{prop_est_psiS} in (\ref{sq_opsq}) with $\mcal L=z=x^\gamma$ and with $A$ sufficiently large.
\end{proof}
\begin{remark}
    In previous work \cite{Wo,SS}, $x$ was restricted to $(\log H)^{1+\varepsilon}<x<(\log H)^{\delta}$ so that the diagonal term is asymptotically smaller than $x^2$. By obtaining a precise asymptotic for the diagonal term, we can remove the condition $x>(\log H)^{1+\varepsilon}$. The condition $x(H)\to \infty$ in Theorem \ref{teo_main} is necessary if we want to restrict ourselves to positive values of polynomials (see \S \ref{sec_trans}), but it can otherwise be dropped, as in Theorem
    \ref{teo_mainabs}.
\end{remark}


\section{The term {$\sum {(\psiabs_P)}^2$}}\label{sec_psisq}
The asymptotic for $\sum_P\psiabs_P(x)^2$ requires asymptotics for both the diagonal and the non-diagonal terms. These are proved, respectively, in  Proposition~\ref{prop_est_diag} and Proposition~\ref{prop_nondiag}.

\subsection{The diagonal term}\label{sec_diag}

Writing $P(t)=c_0+c_1t+\ldots +c_dt^d$, we have 
\begin{align}\label{eq_split}\sum_{P\in \pol d H}\sum_{\substack{m\le x\\P(m)\neq 0}}\Lambda(\ord{P(m)})^2
&=
\sum_{\substack{c_1,\ldots,c_d\in [-H,H]\\c_d>0}}\,\sum_{m\le x}\sum_{\substack{c_0\in [-H,H]\\c_0+c_1m+\ldots+c_dm^d\neq 0}} \Lambda\left(\ord{c_0+c_1m+\ldots+c_dm^d}\right)\,.
\end{align}
In what follows, we fix $c_1,\ldots,c_d\in [-H,H]$ and $m\le x$, and we set $$N\coloneqq c_1m+\ldots+c_dm^d\,\qquad\text{and}\qquad M\coloneqq \max\{\ord{-H+N},\ord{H+N}\}\,.$$ We then derive asymptotics for $\sum\limits_{\substack{c_0\in [-H,H]\\c_0+N\neq 0}} \Lambda(\ord{c_0+N})^2$. First, we decompose this sum as
\begin{align*}
\sum_{\substack{c_0\in [-H,H]\\c_0+N\neq 0}} \left(\Lambda^2-\Lambda\cdot \log \right)(\ord{c_0+N})\,+\,\sum_{\substack{c_0\in [-H,H]\\c_0+N\neq 0}} \left(\Lambda\cdot \log\right)(\ord{c_0+N})\end{align*}
and show that the first sum on the right-hand side is asymptotically small.
\begin{lemma}\label{lemma_d1} For any $H>1$ and $N\in \Z$  we have 
    $$\sum_{\substack{c_0\in [-H,H]\\c_0+N\neq 0}}\left( \Lambda(\ord{c_0+N})^2-\Lambda(\ord{c_0+N})\log (\ord{c_0+N})\right)=O\left( M^{1/2}\log M\right)\,,$$ where the implied constant depends only on $d$.
\end{lemma}
\begin{proof}
     Since $\Lambda $ is supported on prime powers, we consider only those $c_0$ for which $\ord{c_0+N}=\ell^a$ with  $\ell\ll M^{1/a}$  a prime and $a\ll \log M$. Thus, the sum is 
\begin{align*}
    \ll\sum\limits_{2\le a\ll \log M}\sum\limits_{\substack{\ell\ll M^{1/a}\\ \ell\text{ prime}}}a(\log \ell)^2\sum\limits_{\substack{c_0\in [-H,H]}} \mathbb 1_{\ord{c_0+N}=\ell^a}
    &
    \ll \sum\limits_{2\le a\ll \log M}a\sum\limits_{\substack{\ell\ll M^{1/a}\\ \ell\text{ prime}}}  (\log \ell)^2\,,
    \end{align*}
where the inequality $\sum_{c_0}\mathbb 1_{\ord{c_0+N}=\ell^a}\ll 1$ is justified by the fact that there are at most two values of $c_0$ such that $\ord{c_0+N}=\ell^a$. Then we conclude that the sum above is
    \begin{align*}
    &\ll 2M^{1/2}\log (M^{1/2})+O\left(M^{1/3}(\log M)^2\right)=O\left(M^{1/2}\log M\right)\,.\qedhere
\end{align*}
\end{proof}

\begin{lemma}\label{lemma_d2}
     For any $H,A>1$ and $N\in \Z$ we have $$\sum_{\substack{c_0\in [-H,H]\\c_0+N\neq 0}}\left(\log\cdot \Lambda\right)(\ord{c_0+N})=
     \sum_{\substack{c_0\in [-H,H]\\c_0+N\neq 0}}\left(\log M+O(\log \log H)\right)\Lambda(\ord{c_0+N})\,+\, O\left(\frac{M(\log M)^2}{(\log H)^A}\right)\,$$ where the implied constant depends only on $d$.
\end{lemma}
\begin{proof}
    Define $$\mathscr A_{A,N}\coloneqq \left\{c_0\in [-H,H]: 0<\ord{c_0+N}\le \frac{M}{(\log H)^A} \right\}\qquad\text{and}\qquad \mathscr B_{A,N}\coloneqq [-H,H]\setminus \mathscr A_{A,N}\,.$$
Observe that $\#\mathscr A_{A,N}=O\left( \frac{M}{(\log H)^A}\right)$ and that in $\mathscr B_{A,N}$ we have $\frac{M}{(\log H)^A}\le \ord{c_0+N}\le M$, which implies $\log (\ord{c_0+N})=\log M+O(\log \log H)$.
By separating the sum over $c_0\in [-H,H]$ into two sums over $\mathscr A_{A,N}$ and $\mathscr B_{A,N}$, we obtain
\begin{align*}
    \sum_{\substack{c_0\in [-H,H]\\c_0+N\neq 0}}\log(\ord{c_0+N})&\,\Lambda(\ord{c_0+N})\\
    &=
    \sum_{c_0\in \mathscr B_{A,N}}\log(\ord{c_0+N})\Lambda(\ord{c_0+N})\,+\, \sum_{c_0\in \mathscr A_{A,N}}\log(\ord{c_0+N})\Lambda(\ord{c_0+N})\\
    &=
    \sum_{c_0\in \mathscr B_{A,N}}\left(\log M+O(\log \log H)\right)\Lambda(\ord{c_0+N})\,+\, O\left(\frac{M(\log M)^2}{(\log H)^A} \right)\\
        &=\sum_{\substack{c_0\in [-H,H]\\c_0+N\neq 0}}\left(\log M+ O(\log \log H)\right)\Lambda(\ord{c_0+N})
        \\&
        -\,\sum_{c_0\in \mathscr {A}_{A,N}}\left(\log M+O(\log \log H)\right)\Lambda(\ord{c_0+N})+O\left(\frac{M(\log M)^2}{(\log H)^A}\right)\\
    &=
     \sum_{\substack{c_0\in [-H,H]\\c_0+N\neq 0}}\left(\log M +O(\log \log H)\right)\Lambda(\ord{c_0+N})\,+\, O\left(\frac{M(\log M)^2}{(\log H)^A}\right)\,.\qedhere
\end{align*}
\end{proof}

We rewrite the Prime Number Theorem in a convenient form.
\begin{lemma}\label{lemma_s1} Let $H>0$ and $N\in \Z$. Then, for every $A>0$ we have $$\sum_{\substack{c_0\in [-H,H]\\c_0+N\neq 0}}\Lambda\left(\ord{c_0+N}\right)= {2H}+O\left(\frac {M} {(\log M)^A}\right)\,,$$
where the implied constant depends only on $A$.
\end{lemma}

\begin{prop}\label{prop_est_diag_gen}
    Let $\delta,H>1$, $d\ge 2$ and let $x\le (\log H)^\delta$. Then, for all integers $0<m\le x$ and $c_1,\ldots,c_d\in [-H,H]$ with $c_d>0$ we have $$\sum_{\substack{c_0\in [-H,H]\\P(m)\neq 0}}\Lambda(\ord{P(m)})^2=2H(\log H)+O\left(H(\log\log H)\right)\,,$$
    where we write $P(t)$ for $c_dt^d+\ldots+c_1t+c_0$.
      In particular, the implied constant depends only on $\delta $ and not on $c_1,\ldots,c_d$.
\end{prop}
\begin{proof}
     Applying Lemmas~\ref{lemma_d1}--\ref{lemma_s1}, we obtain
     \begin{align*}
        \sum_{\substack{c_0\in [-H,H]\\c_0+N\neq 0}}\Lambda(\ord{c_0+N})^2
        &=
        \sum_{\substack{c_0\in [-H,H]\\c_0+N\neq 0}} \left(\Lambda^2-\Lambda\cdot \log \right)(\ord{c_0+N})+O\left(M^{1/2}(\log M)^2\right)\\
        &=\left(\log M+O\left(\log \log H\right)\right)\sum_{\substack{c_0\in [-H,H]\\c_0+N\neq 0}}\Lambda(\ord{c_0+N})+O\left(\frac M{(\log M)^A}\right)\\  
        &=
        2H\log M+O\left(\frac M{(\log M)^A}+H\log\log H\right)\,.
    \end{align*}
    Since $N\ll m^dH$ and $m\le x\le (\log H)^\delta$, we have $\frac{M}{(\log M)^A}\ll\frac{Hx^d}{(\log H)^A}$, so the error term $O\left(\frac M{(\log M)^A}\right)$ can be made asymptotically smaller than $H(\log\log H)$ by taking $A$ big enough.
    The same considerations on $N$ and $m$ also imply that  $\log M=\log H+O\left(\log\log H\right)$, which is enough to conclude. 
\end{proof}
As an immediate corollary, we have the following result.
\begin{prop}\label{prop_est_diag}
    In the setting of Theorem~\ref{teo_mainabs}, we have $$\frac{1}{2^dH^{d+1}}\sum_{P\in \pol d H}\sum_{m\le x}\Lambda(\ord{P(m)})^2=  x(\log H)+O\left(x(\log\log H)\right)   \,.$$
\end{prop}
\begin{proof}
    It suffices to apply Proposition \ref{prop_est_diag_gen} to (\ref{eq_split}).
\end{proof}


\subsection{The non-diagonal term}\label{sec_nondiag}
We improve the estimate in~\cite[\S 4.1]{SS} for the non-diagonal term, obtaining an error term of $O(x\log \log H)$.

A straightforward modification of~\cite[Theorem 3.1]{SS} yields the following result.
\begin{prop} \label{prop_ss2} Let $H,A>1$, $\delta>0$, $d\ge 2$, and let $x\le (\log H)^\delta$. Then, for all integers $0< m_1<m_2\le x$ and $c_d,\ldots, c_2\in [-H,H]$ with $c_d>0$, we have
    $$\sum_{\substack{c_0,c_1\in [-H,H]\\P(m_1),P(m_2)\neq 0}}\Lambda(\ord{P(m_1)})\Lambda(\ord{P(m_2)})=(2H)^2\frac{m_2-m_1}{\phi(m_2-m_1)}+O_A\left(\frac{H^{2}}{(\log H)^A}\right)\,,$$
    where we write $P(t)$ for $c_dt^d+\ldots+c_1t+c_0$. In particular, the implied constants depend only on $A$ and $\delta$, and not on $c_d,\ldots,c_2$.
\end{prop}
\begin{comm}
   This version is obtained by using two Dirichlet kernels in \cite[Lemma 3.4]{SS} instead of $d+1$ kernels. The calculations for the main term are simpler than in \cite{SS}, since here we have a $2\times 2$ linear system in $c_0$ and $c_1$ in which the coefficients of $c_0$ are both $1$.
\end{comm}

By applying Proposition~\ref{prop_ss2} directly, we obtain
\begin{align}\label{eq3}
    \sum\limits_{P\in \pol d H}\sum\limits_{\substack{1\le m_1\neq m_2\le x\\P(m_1),P(m_2)\neq 0}}&\Lambda(\ord{P(m_1)})\Lambda(\ord{P(m_2)})\\
    &    =
    2\sum\limits_{1\le m_1<m_2\le x}\sum_{\substack{c_d,\ldots,c_2\in [-H,H]\\c_d>0}}\sum_{\substack{c_0,c_1\in [-H,H]\\P(m_1),P(m_2)\neq 0}}\Lambda(\ord{P(m_1)})\Lambda(\ord{P(m_2)})\notag\\
    &
    =2^{d+1}H^{d+1}\sum\limits_{1\le m_1<m_2\le x}\frac{m_2-m_1}{\phi(m_2-m_1)}+O_A\left(\frac{x^2H^{d+1}}{(\log H)^A}\right)\,.\notag
\end{align}
Since $x\le (\log H)^\delta$, the error term above is acceptable as long as $A>2\delta$.

\begin{lemma}\label{prop_sv_mix} For $x\ge 1$, we have
    $$\sum\limits_{1\le m_1<m_2\le x}\frac{m_2-m_1}{\phi(m_2-m_1)}=\frac{x^2}{2}\prod_{\substack{\ell\text{ prime}\\\ell=2}}^\infty\left(1+\frac{1}{\ell(\ell-1)}\right)+O(x\log x)\,.$$ 
\end{lemma}
\begin{proof}Letting $t=m_2-m_1$, we obtain
    \begin{align*}
        \sum\limits_{t\le x}(x-t+O(1))\frac t{\phi(t)} &=(x+O(1))\sum\limits_{t\le x}\frac t {\phi(t)}\,-\,\sum\limits_{t\le x}\frac{t^2}{\phi(t)}\,.
    \end{align*}  Note that
    $$\sum_{t\le x}\frac{t}{\phi(t)}=\sum_{t\le x}\sum_{k\mid t}\frac{\mu(k)^2}{\phi(k)}=\sum_{k\le x}\frac{\mu(k)^2}{\phi(k)}\left\lfloor \frac x k\right\rfloor\ll x\sum_{k\le x}\frac{\mu(k)^2}{k\phi(k)}\ll x\,,$$ and thus, by partial summation, $$\sum_{t\le x}\frac 1 {\phi(t)}\ll \log x\qquad\qquad\text{and}\qquad\qquad\sum_{t> x}\frac{1}{t\phi(t)}\ll\frac 1 x\,.$$ Hence,     
    \begin{align*}
        \sum_{t\le x}\frac{t^2}{\phi(t)}&=
        \sum\limits_{t\le x}t\sum\limits_{k\mid t}\frac {\mu(k)^2}{\phi(k)}
        =
        \sum\limits_{k\le x}\frac{\mu(k)^2k}{\phi(k)}\sum\limits_{r\le x/k}r
        =
        \sum\limits_{k\le x}\frac{\mu(k)^2k}{\phi(k)}\frac{x^2}{2k^2}+O\left(x\sum\limits_{k\le x}\frac{1}{\phi(k)}\right)\\
        &=
        \frac{x^2}{2}\sum\limits_{k\le x}\frac{\mu(k)^2}{k\phi(k)}+O\left(x\log x\right)
        =
        \frac{x^2}{2}\sum\limits_{k=1}^\infty\frac{\mu(k)^2}{k\phi(k)}+O\left(x\log x\right)\\
         &=
        \frac{x^2}{2}\prod\limits_{\substack{\ell=2\\\ell\text{ prime}}}^\infty\left(1+\frac 1{\ell(\ell-1)}\right)+O\left(x\log x\right)
        \,,
        \end{align*}
        and similarly \[
            \sum_{t\le x}\frac t{\phi(t)}
            =
    x\prod\limits_{\substack{\ell=2\\\ell\text{ prime}}}^\infty\left(1+\frac 1{\ell(\ell-1)}\right)+O\left(\log x\right)\,.\qedhere
        \]
        \end{proof}


\begin{prop}\label{prop_nondiag}
    In the setting of Proposition~\ref{prop_est_psi_sq}, we have
    $$\frac{1}{2^dH^{d+1}}\!\sum\limits_{P\in \pol d H}\sum\limits_{\substack{1\le m_1\neq m_2\le x\\P(m_1),P(m_2)\neq 0}}\!\!\!\!\!\! \Lambda(\ord{P(m_1)})\Lambda(\ord{P(m_2)})
    =x^2\!\!\!\prod_{\substack{\ell <\mcal L\\\ell\text{ prime}}}\!\!\left(1+\frac 1 {\ell(\ell-1)}\right)+O\!\left(x\log\log H+\frac{x^2}{\mcal L}\right)\,.$$
\end{prop}
\begin{proof}
    Combine (\ref{eq3}) with Lemma~\ref{prop_sv_mix} to write the sum on the left-hand side as $$\frac{1}{2^dH^{d+1}}\sum\limits_{P\in \pol d H}\sum\limits_{\substack{1\le m_1\neq m_2\le x\\P(m_1),P(m_2)\neq 0}}\Lambda(\ord{P(m_1)})\Lambda(\ord{P(m_2)})
    =x^2\prod_{\substack{\ell\text{ prime}\\\ell=2}}^\infty\left(1+\frac 1 {\ell(\ell-1)}\right)+O\left(x\log\log H\right)\,.$$ By~\cite[Lemma 4.5]{SS}, truncating the product above at $\ell<\mcal L$ introduces the error term $O\left(\frac{x^2}{\mcal L}\right)$.
\end{proof}


\section{The term $\sum_P{\mathfrak S_P}^2$}\label{sec_Ssq}

\subsection{Hooley's neutralisers}
For $w>1$, let $\mcal P(w)\coloneqq\prod_{\text{prime } \ell<w}\ell$. We recall Brun's sieve~\cite[Lemma 6.3]{IK}.
\begin{lemma}[Brun's sieve]\label{lemma_sieve}
    Let $\kappa>0$ and $y>1$. There are two sequences of real numbers $(\lambda_k^+),\,(\lambda_k^-)$ depending on $\kappa,y$, with $$
         \lambda_1^-=\lambda_1^+=1,\qquad\qquad -1\le \lambda_k^+,\lambda_k^- \le1\quad\text{ for }k< y,\qquad\qquad   \lambda_k^+=\lambda_k^-=0\quad\text{ for } k\ge y\,$$
         and such that for any $n>1$ we have $$\sum_{k\mid n}\lambda_k^-\le 0\le \sum_{k\mid  n}\lambda^+_k\,.$$ Furthermore, for every multiplicative function $h\colon\N\to [0,1)$ for which there is $K>0$ such that for every $2\le y_1<y_2\le y$ we have \begin{equation}\label{cond_Brun}
        \prod_{\substack{\ell\text{ prime}\\y_1\le \ell<y_2}}(1-h(\ell))^{-1}\le \left(\frac{\log y_2}{\log y_1}\right)^\kappa\left(1+\frac{K}{\log y_1}\right)\,,
    \end{equation} then for every $w<y$ we have
    \begin{equation}
        \label{concl_sieve}\sum_{k\mid \mcal P(w)}\lambda^\pm_k h(k)=\prod_{\substack{\ell<w\\\ell\text{ prime}}}(1-h(\ell))\cdot \left(1+O\left(e^{-\frac{\log y}{\log w}}\right)\right)\,.
    \end{equation}
\end{lemma}
We use~\cite[Lemma 2.1]{Wi}, which gives a flexible version of Hooley's neutralisers, originally introduced in~\cite{Ho}. Note that~\cite[Lemma 2.1]{Wi} only proves the upper bound, but the lower bound can be proved in the same way.
\begin{lemma}[Hooley's neutralisers]\label{lemma_hooley}
    Let $f\colon \N\rightarrow [0,1]$ be a multiplicative function. Let $(\lambda_k^+)$ and $(\lambda_k^-)$ be sequences of real numbers satisfying, for every $n\in \N$, $$\sum_{k\mid n }\lambda_k^-\le\mathbb 1_{n=1} \le\sum_{k\mid n}\lambda_k^+\,.$$ Let $\hat f(n)\coloneqq \prod_{\text{prime }\ell\mid n}(1-f(\ell))$. Then for every square-free $n\in \N$ we have $$\sum\limits_{\substack{k\mid n}}\lambda^-_k\hat f(k)\le f(n)\le\sum\limits_{\substack{k\mid n}}\lambda^+_k\hat f(k)\,.$$
\end{lemma} We shall apply this with $n=\mcal P(z)$ and $f(n)=\mfS_P(z)$. For $(\lambda_k^\pm)$ we take Brun sieve weights from Lemma~\ref{lemma_sieve}. This has the advantage of involving only small values of $k$, since $\lambda_k^\pm$ are supported on $[1,y]$.
\begin{lemma}\label{lemma_mult} Let $g\colon \Z[t]\times \N\rightarrow \C$ be a function such that\begin{itemize}
    \item for any fixed $P\in \Z[t]$, $g_P\colon\N\rightarrow \C$ is a multiplicative function;
    \item for any fixed $k\in \N$, the value of $g_P(k)$ depends only on the residue class of $P$ in $(\Z/k\Z)[t]$.
\end{itemize}
Fix $d\ge 1$. If we define $G\colon \N\rightarrow \C$ as $$G(k)\coloneqq\sum_{\substack{P_0\in (\Z/k\Z)[t] \\\deg P_0\le d}}g_{P_0}(k)\,,$$ then $G$ is multiplicative.
\end{lemma}
\begin{proof}The proof follows directly from the Chinese remainder theorem.
\end{proof}

\subsection{Calculations for $\sum_{P}\mathfrak S_P^2$}
\begin{lemma}\label{lemma_sum1}
    For $k,d\in \N$ with $k$ square-free we have $$\sum_{\substack{P_0\in (\Z/k\Z)[t]\\\deg P_0\le d}}\prod_{\substack{\ell\mid k\\\ell\text{ prime}}}\left(2\frac{\omega_{P_0}(\ell)}{\ell}-\frac{\omega_{P_0}(\ell)^2}{\ell^2}\right)=\prod_{\substack{\ell|k\\\ell\text{ prime}}}\left(2\ell^d-2\ell^{d-1}+\ell^{d-2}\right)\,.
    $$\end{lemma}
    \begin{proof}
        This follows from Lemma~\ref{lemma_mult}, together with
\begin{equation}\label{sum1}
    \sum_{\substack{P\in (\Z/\ell\Z)[t]\\\deg P\le d}}\omega_P(\ell)=\ell^{d+1}\qquad\text{ and }\qquad
    \sum_{\substack{P\in (\Z/\ell\Z)[t]\\\deg P\le d}}\omega_P(\ell)^2=\ell^{d}(2\ell-1)\,,
\end{equation} which are proved in \cite[\S2]{SS}.
\end{proof}
\begin{proof}[Proof of Proposition~\ref{prop_est_Ssq}] 
By (\ref{def_tbh}) we have
$$\sum_{P\in \pol d H}\mfS_P^2(z)=\left(\prod_{\substack{\ell< z\\\ell\text{ prime}}}\frac {\ell^2} {(\ell-1)^2}\right)\sum_{P\in \pol d H}\prod_{\substack{\ell\mid \mcal{P}(z)\\\ell\text{ prime}}}\left(1-\frac {\omega_P(\ell)} \ell\right)^2\,.$$
  We apply Lemma~\ref{lemma_hooley} with $n=\mcal P(z)$, $f(k)=\prod_{\ell\mid k}\left(1-\frac{\omega_P(\ell)} \ell\right)^2$ and $\lambda_k^+$ as in Lemma~\ref{lemma_sieve}, with $y$ to be chosen later. Then $\hat f(k)=\prod_{\ell\mid k}\left(2\frac{\omega_P(\ell)} \ell-\frac{\omega_P(\ell)^2}{\ell^2}\right)$, and
 \begin{align*}
     \sum_{P\in \pol d H}\prod_{\substack{\ell\mid \mcal{P}(z)\\\ell\text{ prime}}}\left(1-\frac {\omega_P(\ell)} \ell\right)^2&\le \sum_{\substack{k\mid \mcal{P}(z)}}\lambda_k^+\sum_{P\in \pol d H}\prod_{\substack{\ell\mid k\\\ell\text{ prime}}}\left(2\frac{\omega_P(\ell)} \ell-\frac{\omega_P(\ell)^2}{\ell^2}\right)\,.
 \end{align*}
 Since the value of $\prod_{\substack{\ell\mid k}}\left(2\frac{\omega_P(\ell)} \ell-\frac{\omega_P(\ell)^2}{\ell^2}\right)$ depends only on the residue class of $P$ in $(\Z/k\Z)[t]$, we write the sum as
 \begin{align*}      
     &
      \sum_{\substack{k\mid \mcal{P}(z)}}\lambda_k^+\sum_{\substack{P_0\in (\Z/k\Z)[t]\\\deg P_0\le d}}\Bigg(\prod_{\substack{\ell\mid k\\\ell\text{ prime}}}\left(2\frac{\omega_{P_0}(\ell)} \ell-\frac{\omega_{P_0}(\ell)^2}{\ell^2}\right)\Bigg)\sum_{\substack{P\in \pol d H\\P\equiv P_0\,(k)}}1\\
     &=
      \sum_{\substack{k\mid \mcal{P}(z)}}\lambda_k^+\sum_{\substack{P_0\in (\Z/k\Z)[t]\\\deg P_0\le d}}\Bigg(\prod_{\substack{\ell\mid k\\\ell\text{ prime}}}\left(2\frac{\omega_{P_0}(\ell)} \ell-\frac{\omega_{P_0}(\ell)^2}{\ell^2}\right)\Bigg) \left(\frac{2^dH^{d+1}}{k^{d+1}}+O\left(\frac{H^d}{k^d}+1\right)\right)\,.\\
 \end{align*}
By Lemma~\ref{lemma_sieve} we have $\ord{\lambda_k^+}\le 1$ and $\lambda_k^+=0$ for $k\ge y$, hence the error term is bounded by
\begin{align*}
    \sum_{\substack{k\text{ square-free}\\k<y}}&\left(\frac{H^d}{k^d}+1\right)\sum_{\substack{P_0\in (\Z/k\Z)[t]\\\deg P_0\le d}}\prod_{\substack{\ell\mid k\\\ell\text{ prime}}}\left(2\frac{\omega_{P_0}(\ell)}\ell-\frac{\omega_{P_0}(\ell)^2}{\ell^2}\right)
    \ll
    \sum_{\substack{k\text{ square-free}\\k<y}}\left(\frac{H^d}{k^d}+1\right)\sum_{\substack{P_0\in (\Z/k\Z)[t]\\\deg P_0\le d}}2^{\omega(k)}\\
    &\ll H^d \sum_{\substack{k\text{ square-free}\\k<y}}k2^{\omega(k)}+\sum_{\substack{k\text{ square-free}\\k<y}}k^{d+1}2^{\omega(k)}\ll H^d y^2\log y+y^{d+2}\log y
     \end{align*}
due to the standard estimate $\sum\limits_{k\le y} 2^{\omega(k)}=O(y\log y)$.
By Lemma \ref{lemma_sum1}, the main term equals \begin{align*}
    & 2^dH^{d+1}\sum_{\substack{k\mid \mcal{P}(z)}}\lambda_k^+\prod_{\substack{\ell\mid k\\\ell\text{ prime}}}\frac {2\ell^2-2\ell+1} {\ell^3} \,.
\end{align*}
We apply Lemma \ref{lemma_sieve} with $h(\ell)=\frac{2\ell^2-2\ell+1}{\ell^3}$ and $\kappa=2$. Assumption (\ref{cond_Brun}) follows directly from Mertens' theorem. From (\ref{concl_sieve}) we obtain
\begin{align*}
 \sum_{\substack{k\mid \mcal{P}(z)}}\lambda_k^+\prod_{\substack{\ell\mid k\\\ell\text{ prime}}}\frac {2\ell^2-2\ell+1} {\ell^3}=\prod_{\substack{\ell<z\\\ell\text{ prime}}}\frac{(\ell-1)(\ell^2-\ell+1)}{\ell^3}\left(1+O\left( \exp\left({-\frac{\log y}{\log z}}\right)\right)\right)\,.\\
\end{align*}
Taking $y=\frac{H^{1/2}}{(\log H)^A}$ and assembling all the pieces together gives
\begin{align*}
    \frac{1}{2^dH^{d+1}}\sum_{P\in \pol d H}&\mathfrak S_P(z)^2\\
    &\le 
    \prod_{\substack{\ell< z\\\ell\text{ prime}}}\frac {\ell^2} {(\ell-1)^2}\left(\prod_{\substack{\ell<z\\\ell\text{ prime}}}\frac{(\ell-1)(\ell^2-\ell+1)}{\ell^3}\left(1+O\left(H^{-\frac 1 {2\log z}}\right)\right)+O\left(\frac{\log H}{(\log H)^{2A}}\right)\right)\\
    &
    =\prod_{\substack{\ell<z\\\ell\text{ prime}}}\left(1+\frac{1}{\ell(\ell-1)}\right)+O\left((\log z)^2H^{-\frac{1}{2\log z}}+\frac{(\log z)^2}{(\log H)^{2A-1}}\right)\,.
\end{align*}
The conclusion is now immediate by repeating the same process with $\lambda ^-$ instead of $\lambda^+$.
\end{proof}

\section{The term $\sum \psiabs_P\mathfrak S_P$}\label{sec_psiS}
 Given $X>0$ and $q,b\in \N$, define $$E(X;q,b)\coloneqq \sum\limits_{\substack{0<n\le   X\\n\equiv b\bmod q}}\Lambda(n)-\frac{  X}{\phi(q)}\mathbb 1_{\gcd(q,b)=1}\,.$$
 Let us recall Bombieri--Vinogradov's theorem \cite[\S28]{Da}.
 \begin{lemma}\label{lemma_BV}
     Let $A>5$ be fixed. Then for all $X>2$ we have 
             \begin{align*}
                 \sum_{q\le \frac{\sqrt {X}}{(\log   X)^{A}}}\max_{Y\le X}\max_{b\in (\Z/q\Z)^\times} \big|E(Y;q,b)\big|\ll\frac{X}{(\log X)^{A-5}}\,.
            \end{align*}
 \end{lemma} 
\begin{proof}[Proof of Proposition~\ref{prop_est_psiS}]

Write $P(t)=c_dt^d+\ldots+c_0$. By (\ref{def_tbh}) we have
\begin{align*}
\sum_{P\in \pol d H}&\psiabs_P(x)\mfS_P(z)
=\left(\prod_{\substack{\ell< z\\\ell\text{ prime}}}\frac \ell {\ell-1}\right)\sum_{P\in \pol d H}\psiabs_P(x)\prod_{\substack{\ell<z\\\ell\text{ prime}}}\left(1-\frac {\omega_P(\ell)} \ell\right)\\
&=
\left(\prod_{\substack{\ell< z\\\ell\text{ prime}}}\frac \ell {\ell-1}\right)\sum_{m\le x}\sum_{\substack{0<c_d\le H\\\ord{c_1}, \ldots,\ord{c_{d-1}}\le H}}\sum_{\substack{\ord{c_0}\le H\\P(m)\neq 0}}\Lambda\left(\ord{c_0+c_1m+\ldots+c_dm^d}\right)\prod_{\substack{\ell<z\\\ell\text{ prime}}}\left(1-\frac {\omega_P(\ell)} \ell\right)
\,.\end{align*}
Let $N\coloneqq c_1m+\ldots+c_dm^d$, so that $P(m)=c_0+N$. We apply Lemma~\ref{lemma_hooley} with $n=\mcal P(z)$, $f(k)=\prod_{\ell\mid k}\left(1-\frac{\omega_P(\ell)} \ell\right)$ and $\lambda_k^+$ as in Lemma \ref{lemma_sieve} with $y=\frac{H^{1/2}}{(\log H)^A}$. Then $\hat f(k)=\frac{\omega_P(k)}{k}$ and
    \begin{align*}
        \sum_{\substack{\ord{c_0}\le H\\c_0+N\neq 0}}\Lambda(\ord{c_0+N})\prod_{\substack{\ell<z\\\ell\text{ prime}}}&\left(1-\frac {\omega_P(\ell)} \ell\right)
        \le
        \sum_{k\mid \mcal{P}(z)}\frac{\lambda^+_k}{k}\sum_{\substack{\ord{c_0}\le H\\c_0+N\neq 0}}\Lambda(\ord{c_0+N})\,\omega_P(k)\,.
        \end{align*}
        We can introduce into the sum over $c_0$ the condition $\gcd(c_0+N,k)=1$ since
        
        \begin{align*}
            \sum_{k\mid \mcal{P}(z)}\frac{\ord{\lambda^+_k}}{k}\sum_{\substack{\ord{c_0}\le H,\, c_0+N\neq 0\\\gcd(c_0+N,k)\neq 1}}&\Lambda(\ord{c_0+N})\,\omega_P(k)\le \sum_{k\le y}\sum_{\substack{\ord{c_0}\le H,\, c_0+N\neq 0\\\gcd(c_0+N,k)\neq 1}}\Lambda(\ord{c_0+N})\\
            &\ll
            \sum_{k\le y}\sum_{\substack{\ell\mid k\\\ell\text{ prime}}}\sum_{\alpha\ll\log H}(\log H)  \#\{c_0\in [1,H]: c_0+N=p^\alpha\}\\
            &\ll(\log H)^2 \sum_{k\le y}\omega(k)\ll y(\log y)(\log H)^2\ll \frac{H^{1/2}(\log H)^3}{(\log H)^A}
            \,,
        \end{align*}
        where we have used the standard estimate $\omega(k)\ll \log k$ together with $\ord{\lambda_k^+}\le 1$  and $\lambda_k^+=0 $ for $k\ge y$ from Lemma \ref{lemma_sieve}.
        Thus,
        \begin{align*}
        \sum_{m\le x}\sum_{\substack{0<c_d\le H\\\ord{c_1}, \ldots,\ord{c_{d-1}}\le H}}\sum_{k\mid \mcal P(z)}\frac{\ord{\lambda^+_k}}{k}\sum_{\substack{0< c_0\le H,\, c_0+N\neq 0\\\gcd(c_0+N,k)\neq 1}}\Lambda(\ord{c_0+N})\,\omega_P(k)\ll \frac{xH^{d+1/2}(\log H)^3}{(\log H)^A}    
        \end{align*}
        Going back to the main term,
        \begin{align*}
        \sum_{k\mid \mcal{P}(z)}\frac{\lambda^+_k}{k}\sum_{\substack{\ord{c_0}\le H\\c_0+N\neq 0\\\gcd(P(m),k)=1}}\Lambda(\ord{c_0+N})\,\omega_P(k)
        &=
        \sum_{k\mid \mcal{P}(z)}\frac{\lambda^+_k}{k}\sum_{\substack{\ord{c_0}\le H\\c_0+N\neq 0\\\gcd(P(m),k)=1}}\Lambda(\ord{c_0+N})\sum_{\eta=1}^k\mathbb 1_{P(\eta)\equiv 0\bmod k}
        \,.
        \end{align*}
        Let $a\coloneqq c_1\eta+\ldots +c_d\eta^d$, so that $P(\eta)=c_0+a$. Recalling that $P(m)=c_0+N$, the two conditions $P(\eta)\equiv 0\bmod k$ and $\gcd(P(m),k)=1$ can be rewritten as $\gcd(N-a,k)=1$ and $c_0\equiv -a\bmod k$. Thus, we can write
        \begin{align*}
             \,&\sum_{k\mid \mcal{P}(z)}\frac{\lambda^+_k}{k}\sum_{\substack{\eta=1\\\gcd(N-a,k)=1}}^k\sum_{\substack{-H\le c_0\le H\\c_0+N\neq 0\\c_0\equiv -a\bmod k}}\Lambda(\ord{c_0+N})
             =
             \sum_{k\mid \mcal{P}(z)}\frac{\lambda^+_k}{k}\sum_{\substack{\eta=1\\\gcd(N-a,k)=1}}^k\sum_{\substack{-H+N\le n\le H+N\\n\neq 0\\n\equiv N -a\bmod k}}\Lambda(\ord{n})\\
             &=
              \sum_{k\mid \mcal{P}(z)}\frac{\lambda^+_k}{k}\sum_{\substack{\eta=1\\\gcd(N-a,k)=1}}^k\left[\frac{2H+O(1)}{\phi(k)}\mathbb 1_{\gcd(N-a,k)=1}+E(\ord{H+N};k,N-a)-E(\ord{N-H};k,N-a)\right]\,.
             \end{align*}
             Since $\ord{N-H},\ord{H+N}\le Hm^d$,  we have  $$\big|E(\ord{N-H};k,N-a)\big|,\big|E(\ord{N+H};k,N-a)\big|\le \max_{Y\le Hm^d}\max_{b\in (\Z/k\Z)^\times}\big|{E(Y;k,b)}\big|\,.$$ By Lemma \ref{lemma_BV} with $X=(d+1)Hm^d$ we obtain 
             \begin{align*}
                 \sum_{m\le x}\sum_{\substack{0<c_d\le H\\\ord{c_1}, \ldots,\ord{c_{d-1}}\le H}}&
              \sum_{k\mid \mcal{P}(z)}\frac{\ord{\lambda^+_k}}{k}\sum_{\substack{\eta=1\\\gcd(N-a,k)=1}}^k\big|E(\ord{N+H};k,N-a)-E(\ord{N-H};k,N-a)\big|\\
              &\ll
                 \sum_{m\le x}\sum_{\substack{0<c_d\le H\\\ord{c_1}, \ldots,\ord{c_{d-1}}\le H}} \sum_{k<y}\big|E(\ord{N+H};k,N-a)\big|+\big|E(\ord{N-H};k,N-a)\big|\\
              &\ll
              \sum_{m\le x}\sum_{\substack{0<c_d\le H\\\ord{c_1}, \ldots,\ord{c_{d-1}}\le H}}\frac{Hm^d}{(\log H)^{A-5}}
              \ll
              \frac{H^{d+1}x^{d+1}}{(\log H)^{A-5}}\,.
             \end{align*}
        Moving to the main term, 
        \begin{align*}
            2H\sum_{m\le x}\sum_{\substack{0<c_d\le H\\\ord{c_1}, \ldots,\ord{c_{d-1}}\le H}}
              \sum_{k\mid \mcal{P}(z)}&\frac{{\lambda^+_k}}{k\phi(k)}\sum_{\substack{\eta=1\\\gcd(N-a,k)=1}}^k 1\\ 
              &=
              2H\sum_{m\le x}\sum_{k\mid \mcal{P}(z)}\frac{{\lambda^+_k}}{k\phi(k)}\#\left\{\begin{array}{l}
                c_1,\ldots,c_{d-1}\in [-H,H]\\
                c_d\in [1,H]\\
                \eta\in [1,\ldots, k]\end{array}\:\middle|\;\begin{array}{c} \gcd(N-a,k)=1
            \end{array}\right\}\\ 
              \end{align*}
              Since $N-a=(m-\eta)\left(c_1+c_2\frac{m^2-\eta^2}{m-\eta}+\ldots+c_d\frac{m^d-\eta^d}{m-\eta}\right)$, the condition $\gcd(N-a,k)=1$ is equivalent to $$\gcd(m-\eta,k)=1\qquad \text{and}\qquad \gcd\left(c_1+c_2\frac{m^2-\eta^2}{m-\eta}+\ldots+c_d\frac{m^d-\eta^d}{m-\eta},k\right)=1\,.$$
               To count the cardinality of the set above, note that one can freely choose $b_1,b_2\in (\Z/k\Z)^\times$ in $\phi(k)^2$ ways and impose the two conditions $$m-\eta\equiv b_1\bmod k\qquad \text{and}\qquad c_1+c_2\frac{m^2-\eta^2}{m-\eta}+\ldots+c_d\frac{m^d-\eta^d}{m-\eta}\equiv b_2\bmod k\,.$$ Then the first condition is satisfied for exactly one choice of $\eta$, while the second one, since $c_1$ has coefficient $1$, is satisfied by $\frac{2^{d-1}H^d} k+O\left(H^{d-1}\right)$ choices of $c_1,\ldots,c_d$. Therefore,
              \begin{align*}
              2H\sum_{m\le x}\sum_{k\mid \mcal{P}(z)}\frac{{\lambda^+_k\phi(k)}}{k}\left[\frac{2^{d-1}H^d}{k}+O\left(H^{d-1}\right)\right]
              =
              2^dH^{d+1}\left(x+O(1)\right)\sum_{k\mid \mcal{P}(z)}\lambda^+_k\frac{\phi(k)}{k^2}+O\left(\frac{xH^{d+1/2}}{(\log H)^A}\right)\,.
        \end{align*}
    We apply Lemma \ref{lemma_sieve} to the main term with $h(\ell)=\frac{\ell-1}{\ell^2}$ and $\kappa=1$. Assumption (\ref{cond_Brun}) follows by Mertens' theorem. From (\ref{concl_sieve}) we obtain

\begin{align*}
 \sum_{\substack{k\mid \mcal{P}(z)}}\lambda_k^+\prod_{\substack{\ell\mid k\\\ell\text{ prime}}}\frac {\ell-1} {\ell^2}=\prod_{\substack{\ell<z\\\ell\text{ prime}}}\frac{(\ell^2-\ell+1)}{\ell^2}\left(1+O\left( H^{-\frac{1}{2\log z}}\right)\right)\,.\\
\end{align*}
    Assembling all the pieces together gives
    \begin{align*}
        \frac{1}{2^dH^{d+1}x}\sum_{P\in \pol d H}&\psiabs_P(x)\mfS_P(z)&\\
        &\le
        \prod_{\substack{\ell< z\\\ell\text{ prime}}}\frac \ell {\ell-1}\left[\prod_{\substack{\ell<z\\ \ell\text{ prime}}}\left(\frac{\ell^2-\ell+1}{\ell^2}\right)\left(1+O\left(H^{-\frac 1 {2\log z}}\right)\right)+O\left(\frac{(\log H)^{\delta d}}{(\log H)^{A-5}}\right)\right]\\
        &
        =\prod_{\substack{\ell<z\\ \ell\text{ prime}}}\frac{\ell^2-\ell+1}{\ell(\ell-1)}+O\left((\log z)\,{H^{-\frac 1 {2\log z}}}+\frac{\log z}{(\log H)^{A-\delta d-5}}\right)\,.
    \end{align*}
    The conclusion follows by repeating the same process with $\lambda^-$ instead of $\lambda^+$.
\end{proof}


\section{Transition from $\psiabs$ to $\psi$}\label{sec_trans}

\begin{lemma}\label{lemma_neg}In the context of Theorem \ref{teo_main}, we have
    \begin{align*}
        \sum_{P\in \pol d H}\Bigg|\sum_{\substack{m\le x\\P(m)<0}}\Lambda(-P(m))\Bigg|^2\asymp H^{d+1}\Big(x(\log x)+(\log H)(\log x)\Big) \,.
    \end{align*}
\end{lemma}
\begin{proof}
    We expand the square to get
    \begin{align*}
        \sum_{P\in \pol d H}\sum_{\substack{m\le x\\P(m)<0}}\Lambda(-P(m))^2\,+2\sum_{P\in \pol d H}\sum_{\substack{m_1<m_2\le x\\P(m_1),P(m_2)<0}}\Lambda(-P(m_1))\Lambda(-P(m_2))\,.
    \end{align*}
    Writing $P(m)=c_dm^d+\ldots+c_0$, note that $P(m)> c_dm^d-dHm^{d-1}$, and so $P(m)<0$ implies $c_d<\frac{Hd}{m}$. Thus, applying Proposition \ref{prop_est_diag_gen} with $c_1,\ldots,c_{d-1}\in [-H,H]$ and $c_d\in \left(0,\frac{dH}{m}\right)$, we have
     \begin{align*}
         \sum_{\substack{m\le x}}\sum_{\substack{P\in \pol d H\\P(m)<0}} \Lambda(-P(m))^2
         &\le
         \sum_{m\le x}\sum_{\substack{P\in\pol d H\\c_d<\frac {Hd}{m}}}\Lambda(\ord{P(m)})^2
         \ll 
         \sum_{m\le x}\frac{H^{d+1}}{m}(\log H)\\
         &\ll H^{d+1}(\log H)(\log x)\,.
     \end{align*}
     This is acceptable since $\log H\le x\le (\log H)^\delta$ with $\delta\ge 1$.
     Concerning the non-diagonal part, repeating the same idea gives  $$\{P\in \pol d H: P(m_1),P(m_2)<0\}\subseteq \left\{P\in \pol d H: c_d<\frac{Hd}{\min\{m_1,m_2\}}\right\}\,,$$ and thus, by  Proposition \ref{prop_ss2} with $c_2,\ldots,c_{d-1}\in [-H,H]$, $m_1\le m_2$ and $c_d\in \left(0,\frac{Hd}{m_1}\right)$, we have 
     \begin{align*}
         \sum_{P\in \pol d H}\sum_{\substack{m_1<m_2\le x\\P(m_1),P(m_2)<0}}&\Lambda(-P(m_1))\Lambda(-P(m_2))
         \le
         \sum_{m_1<m_2\le x}\sum_{\substack{P\in \pol d H\\c_d<\frac {Hd}{m_1}}}\Lambda(\ord{P(m_1)})\Lambda(\ord{P(m_2)})\\
         &\ll H^{d+1}
         \sum_{m_1<m_2\le x}\frac{m_2-m_1}{m_1\phi(m_2-m_1)}=H^{d+1}\sum_{1\le t\le x}\frac{t}{\phi(t)}\sum_{1\le m_1\le x-t}\frac{1}{m_1}\\&
         \ll
         H^{d+1} \sum_{1\le t\le x}\frac{t}{\phi(t)}(\log x)
         \ll 
         H^{d+1}x(\log x)\,.\qedhere
     \end{align*}
    The lower bound is obtained by similar calculations, since $c_d m^d\le H/2$ and $c_{d-1}\ldots,c_0\in [-H,-H/2]$ imply $P(m)<0$.
\end{proof}

\begin{prop}In the context of Theorem \ref{teo_main}, we have
    \begin{align*}
        \sum_{P\in \pol d H}\ord{\psi_P(x)-x\mfS_P(x)}^2=\sum_{P\in \pol d H}& \ord{\psiabs_P(x)-x\mfS_P(x)}^2\\&+{O\left(H^{d+1}\bigg(x\sqrt{(\log H)(\log x)}+(\log H)\sqrt{x(\log x)}\bigg)\right)}\,.
    \end{align*}
    \begin{proof} By (\ref{def_psiabs}) we can write $$\psi_P(x)-x\mfS_P(x)=\psiabs_P(x)-x\mfS_P(x)-\sum\limits_{\substack{m\le x\\P(m)<0}}\Lambda(-P(m))$$ and by the approximation $(a-b)^2=a^2+O\left(\ord{ab}+\ord b^2\right)$, valid for $a,b\in \R$, the desired sum equals
        \begin{align*}
             \sum_{P\in \pol d H}\ord{\psiabs_P(x)-x\mfS_P(x)}^2&+O\Bigg(\sum_{P\in \pol d H}\ord{\psiabs_P(x)-x\mfS_P(x)}\cdot \Bigg|\sum_{\substack{m\le x\\P(m)<0}}\Lambda(-P(m))\Bigg|\Bigg)\\
             &+O\Bigg(\sum_{P\in \pol d H}\Bigg|\sum_{\substack{m\le x\\P(m)<0}}\Lambda(-P(m))\Bigg|^2\Bigg)       \,.
        \end{align*}
        Also, by Cauchy's inequality, the middle term is at most\begin{align*}
            \le \left(\sum_{P\in \pol d H}\ord{\psiabs_P(x)-x\mfS_P(x)}^2\right)^{1/2}\Bigg(\sum_{P\in \pol d H}\Bigg|\sum_{\substack{m\le x\\P(m)<0}}\Lambda(-P(m))\Bigg|^2\Bigg)^{1/2}\,.
        \end{align*}
        Using Lemma \ref{lemma_neg} together with Theorem \ref{teo_mainabs} leads to the result.
    \end{proof}

\end{prop}


\bibliographystyle{plainurl}
\bibliography{sample}

\end{document}